\documentclass[[final,1p,times]{elsarticle}
\usepackage{amsmath}
\usepackage{amssymb}
\usepackage{url}
\usepackage{amsfonts}
\usepackage{graphicx}
\usepackage{rotating}
\usepackage{floatflt,epsfig}
\usepackage{lineno,hyperref}
\usepackage{enumerate}
\usepackage{colortbl}
\usepackage{array,tabularx,tabulary,booktabs}
\usepackage{longtable}
\usepackage{multirow}
\usepackage{wrapfig}
\usepackage{subcaption}
\newcolumntype{^}{>{\currentrowstyle}}

\journal{arxiv}
\newtheorem{lemma}{Lemma}

\newtheorem{theorem}{Theorem}

\newtheorem{remark}{Remark}
\newcommand{\proof}{\medskip\noindent{\bf Proof.~}}
\begin{document}
\renewcommand{\abstractname}{Abstract}
\renewcommand{\refname}{References}
\renewcommand{\tablename}{Figure.}
\renewcommand{\arraystretch}{0.9}
\thispagestyle{empty}
\sloppy

\begin{frontmatter}
\title{Generalised dual Seidel switching and Deza graphs with strongly regular children}

\author[01,02]{Vladislav~V.~Kabanov\corref{cor1}}
\cortext[cor1]{Corresponding author}
\ead{vvk@imm.uran.ru}

\author[02,03]{Elena~V.~Konstantinova}
\ead{e\_konsta@math.nsc.ru}

\author[01,02,04]{Leonid~Shalaginov}
\ead{44sh@mail.ru}

\address[01]{Krasovskii Institute of Mathematics and Mechanics, S. Kovalevskaja st. 16, Yekaterinburg, 620990, Russia}
\address[02]{Sobolev Institute of Mathematics, Ak. Koptyug av. 4, Novosibirsk, 630090, Russia}
\address[03]{Novosibisk State University, Pirogova str. 2, Novosibirsk, 630090, Russia}
\address[04]{Chelyabinsk State University, Brat'ev Kashirinyh st. 129, Chelyabinsk,  454021, Russia}

\begin{abstract}
A Deza graph $G$ with parameters $(n,k,b,a)$ is a $k$-regular graph with $n$ vertices such that any two  distinct vertices have $b$ or 
$a$ common neighbours, where $b \geqslant a$. The children $G_A$ and $G_B$ of a Deza graph $G$ are defined on the vertex set of $G$ such that every two distinct vertices are adjacent in $G_A$ or $G_B$ if and only if they have $a$ or $b$ common neighbours, respectively. In this paper we present a general approach to dual Seidel switching and investigate Deza graphs whose children  are strongly regular graphs.
\end{abstract}

\begin{keyword} dual Seidel switching \sep Deza graph \sep strictly Deza graph \sep strongly regular graph  
\vspace{\baselineskip}
\MSC[2010] 05C50\sep 05E10\sep 15A18
\end{keyword}
% 05C50 graphs and linear algebra (matrices, eigenvalues, etc.)
% 05E10 graphs and matrices
% 15A18 eigenvalues, singular values, and eigenvectors
\end{frontmatter}

\section{Introduction}\label{sec0}
A Deza graph $G$ with parameters $(n,k,b,a)$ is a $k$-regular connected graph of order $n$ for which the number of common neighbours of two distinct vertices takes just two  values, $b$ or $a$, where $b\geqslant a$. The graph was introduced in~\cite{D94}, and the name was given in~\cite{EFHHH}, where the basics of Deza graph theory were founded and different constructions of Deza graphs were presented. Moreover, it was shown that Deza graphs can be considered in terms of matrices. Suppose $G$ is a graph with $n$ vertices, and $M$ is its adjacency matrix. Then $G$ is a Deza graph with parameters $(n,k,b,a)$ if and only if 
\begin{equation}
M^2=a~A+b~B+k~I\label{child}
\end{equation}
for some symmetric $(0,1)$-matrices $A$, $B$ such that $A+B+I=J$, where $J$ is the all--ones matrix and $I$ is the identity matrix. Note that $G$ is a strongly regular graph if and only if $A$ or $B$ is $M$. 

Suppose that we have a Deza graph with $M$, $A$, and $B$ satisfying~(\ref{child}). Then $A$ and $B$ are adjacency matrices of graphs, and the corresponding graphs $G_A$ and $G_B$ are called  the {\it children} of $G$. If $G$ is a strongly regular graph, then its children are $G$ and its complement. A Deza graph of diameter two which is not a strongly regular graph is called a {\it strictly Deza graph}. A strongly regular graph is called~\emph{imprimitive} if it, or its complement, is disconnected, and \emph{primitive} otherwise. Deza graphs whose children are imprimitive strongly regular graphs are known as divisible design graph~\cite{HKM}.

The~\emph{dual Seidel switching}~\cite{Hem,JS} was used in~\cite{EFHHH} for constructing (strictly) Deza graphs from strongly regular graphs. We call an involutive automorphism of a graph as a \emph{Seidel automorphism} if it interchanges only non-adjacent vertices. 

\begin{theorem}\label{dual}~{\rm{\bf (Dual Seidel switching)}}~{\rm \cite[  Theorem 3.1]{EFHHH}} 
Let $G$ be a strongly regular graph with parameters $(n,k,\lambda,\mu)$, where $k\neq \mu$, $\lambda \neq \mu$. Let $M$ be the adjacency matrix of $G$, and $P$ be a non-identity permutation matrix of the same size. Then $PM$ is the adjacency matrix of a Deza graph $H$ if and only if $P$ represents a Seidel automprphism. Moreover, $H$ is a strictly Deza graph if and only if $\lambda \neq 0$, $\mu \neq 0$.
\end{theorem}

By Theorem~\ref{dual} we have $(PM)^2=M^2$, therefore there exist Deza graphs with primitive strongly regular children. On the conference ``Graphs and Groups, Design and Dynamics'', Yichang, China, August 2019,  W.~Haemers asked the first author 
 whether there exist other constructions or examples of Deza graphs with primitive strongly regular children. 

%it was suggested by W.~Haemers to study Deza graphs whose children are primitive strongly regular graphs. 

In this paper we present some constructions of Deza graphs with strongly regular children using a generalisation of the dual Seidel switching and other matrix  tools. 

The following significant connections between spectra of Deza graphs and spectra of their children were shown in~\cite{AGHKKS}. 

\begin{theorem}~{\rm \cite[Theorem 3.2]{AGHKKS}} \label{children_spec}
Let $G$ be a Deza graph with parameters $(n,k,b,a)$, $b>a$. Let $M,A,B$ be the adjacency matrices of $G$ and its children, respectively. If $\theta_1=k,\theta_2,\ldots,\theta_n$ are the eigenvalues of $M$, then 

\emph{(i)} The eigenvalues of $A$ are
  $$\alpha = \cfrac{b(n-1)-k(k-1)}{b-a},\ \alpha_2 =\cfrac{k-b-\theta_2^2}{b-a},\ 
  \ldots,\ \alpha_n = \cfrac{k-b-\theta_n^2}{b-a}.$$

\emph{(ii)} The eigenvalues of $B$ are
  $$\beta =  \cfrac{a(n-1)-k(k-1)}{a-b},\ \beta_2 = \cfrac{k-a-\theta_2^2}{a-b},\ 
  \ldots,\ \beta_n = \cfrac{k-a-\theta_n^2}{a-b}.$$
\end{theorem}

If $G$ is a strongly regular graph with parameters $(n,k,\lambda,\mu)$ and with the adjacency matrix $M$, then
\begin{equation}\label{M^2}
M^2 = kI+\lambda M + \mu (J-I-M).
\end{equation} 
It is known~\cite[Theorem 1.3.1]{BCN89} and \cite[Theorem 9.1.2]{BH} that $M$ has precisely two distinct restricted eigenvalues $r$ $(r\geq 0)$ and $s$ $(s\leq 0)$. If a Deza graph has strongly regular children, then we have the following theorem as a straightforward corollary of Theorem~\ref{children_spec}. 

\begin{theorem}\label{spec1}
Let $G$ be a Deza graph with parameters $(n,k,b,a)$, $b>a$. Let a Deza child $H$ of $G$ be a strongly regular graph with parameters $(n,k_H,\lambda,\mu)$, and $H$ has restricted eigenvalues $r$, $s$ with multiplicities $f$, $g$. Then the following statements hold.

\emph{(i)} $k_H$ is equal to $ \cfrac{b(n-1)-k(k-1)}{b-a}$ or $ \cfrac{k(k-1)-a(n-1)}{b-a}$.

\emph{(ii)}  $G$ has at most four restricted eigenvalues
$$\rho_{1,2} = \pm \sqrt{k-b-r(b-a)},\quad  \sigma_{1,2} = \pm \sqrt{k-b-s(b-a)}$$ with the multiplicities $f_1,f_2$ and $g_1,g_2$, respectively.

\emph{(iii)}  $f=f_1+f_2$, $g=g_1+g_2$.
\end{theorem}

Note that Theorem~\ref{spec1} is a generalization of a result obtained for divisible design graphs in~\cite{HKM}. The next theorem follows from Theorem~\ref{spec1}.

\begin{theorem}\label{spec2}
Let $M$ be the adjacency matrix of a Deza graph with strongly regular children. If $N$ is the adjacency matrix of one of its children and the equality $M^2=N^2$ holds, then $k=\alpha$, $\rho_{1,2} =\pm r $ and  $\sigma_{1,2}=\pm s$.
\end{theorem}

\section{Generalisation of the dual Seidel switching}

Let $M$ be the adjacency matrix of a strongly regular graph $G$ with parameters $(n,k,\lambda,\mu)$ such that 
\begin{equation}\label{M}
  M= \left(\begin{array}{cc} M_{11} & M_{12}\\ M_{21}& M_{22}\\ \end{array}\right)      
\end{equation}
for some submatrices $M_{11}, M_{12},M_{21},M_{22}$ of $M$, where $M_{11}$ be the adjacency matrix of an induced subgraph $H$ of $G$. 

Let us consider the permutation matrix
\begin{equation}\label{P}
P = \left(\begin{array}{cc} P_{11} & 0\\ 0 & I\\ \end{array}\right),
\end{equation}
where $P_{11}$ is the permutation matrix corresponding to a Seidel automorphism $\varphi$ of $H$, and the matrix
\begin{equation}\label{N}
PMP=\left(\begin{array}{cc} M_{11} & P_{11}M_{12}\\ M_{21}P_{11} & M_{22}\\ \end{array}\right).
\end{equation}

Below we assume that $M$ and $P$ are defined by~{\rm (\ref{M})} and~{\rm (\ref{P})}, respectively.  It is clear that $P^2=I$.

\begin{lemma}\label{MM}
Let $G$ be a strongly regular graph with the adjacency matrix $M$, and $H$ be its induced subgraph with the adjacency matrix $M_{11}$. If there exists an automorphism of $H$ with the permutation matrix $P_{11}$, then 
\begin{equation}\label{PMMP=MM}
P_{11}M_{12}M_{21}P_{11} = M_{12}M_{21}.
\end{equation}
\end{lemma}

\proof Let us consider an element of matrix $M_{12}M_{21}$ indexed by two distinct vertices $u_1,u_2 \in V(H)$. Then, for an automorphism $\varphi$ of $H$, we have $$|N_{H}(u_1)\cap N_{H}(u_2)|=|N_{H}(\varphi(u_1))\cap N_{H}(\varphi(u_2))|.$$
Since $u_1$ is adjacent to $u_2$ in $G$ if and only if $\varphi(u_1)$ is adjacent to $\varphi(u_2)$ in $G$, and $G$ is a strongly regular graph, we have $$|N_{G}(u_1)\cap N_{G}(u_2)|=|N_{G}(\varphi(u_1))\cap N_{G}(\varphi(u_2))|.$$

Therefore,
$$\big |\big (N_{G}(u_1)\setminus N_{H}(u_1)\big)\cap\big(N_{G}(u_2)\setminus N_{H}(u_2)\big)\big |=$$ $$=\big|\big(N_{G}(\varphi(u_1))\setminus N_{H}(\varphi(u_1))\big)\cap \big(N_{G}(\varphi(u_2))\setminus N_{H}(\varphi(u_2))\big)\big|,$$
which gives~(\ref{PMMP=MM}) and completes the proof. \hfill \hfill $\square$
\medskip 

\begin{theorem}  The matrix $PMP$ is the adjacency matrix of a strongly regular graph with the same parameters $(n,k,\lambda,\mu)$ as a strongly regular graph with the adjacency matrix $M$.
\end{theorem}
\proof To show that the matrix $PMP$ corresponds to a strongly regular graph, we need to consider the matrix $(PMP)^2$. By~(\ref{PMMP=MM}) we have
$$(PMP)^2=PM^2P=P\big(kI+\lambda M +\mu (J-I-M)\big)P=$$ $$ =kI+\lambda PMP +\mu (J-I-PMP).$$
By~(\ref{M^2}), this completes the proof. \hfill $\square$

\medskip

The following theorems generalise the dual Seidel switching.

\begin{theorem}~{\rm{\bf (Generalised dual Seidel switching 1)}}\label{main11}\\
Let $G$ be a strongly regular graph with the adjacency matrix $M$, and $H$ be its induced subgraph with the adjacency matrix $M_{11}$. If there exists a Seidel automorphism of $H$ with the permutation matrix $P_{11}$ such that  $P_{11}M_{12}M_{22}=M_{12}M_{22}$, then matrices
$$N_1= \left(\begin{array}{cc} P_{11}M_{11} & M_{12}\\ M_{21}& M_{22}\\ \end{array}\right),\, \, 
\mbox{and} \, \, N_2= \left(\begin{array}{cc} P_{11}M_{11} & P_{11}M_{12}\\ M_{21}P_{11} & M_{22}\\ \end{array}\right)$$
are the adjacency matrices of Deza graphs with strongly regular children. Moreover, $N_2^2=M^{2}$ and $N_1^2=(PMP)^2$.
\end{theorem}

\begin{theorem}~{\rm{\bf (Generalised dual Seidel switching 2)}}\label{main12}\\
Let $G$ be a Deza graph with strongly regular children and the adjacency matrix $M$, and $H$ be its induced subgraph with the adjacency matrix $M_{11}$. If there exists a Seidel automorphism of $H$ with the permutation matrix $P_{11}$ such that $P_{11}M_{11}M_{12}=M_{11}M_{12}$, then matrices
$$N_1= \left(\begin{array}{cc} P_{11}M_{11} & M_{12}\\ M_{21}& M_{22}\\ \end{array}\right),\, \, 
\mbox{and} \, \, N_2= \left(\begin{array}{cc} P_{11}M_{11} & P_{11}M_{12}\\ M_{21}P_{11} & M_{22}\\ \end{array}\right)$$
are the adjacency matrices of Deza graphs with strongly regular children. Moreover, $N_1^2=M^{2}$ and $N_2^2=(PMP)^2$.
\end{theorem}

\begin{proof}
To proof both theorems we calculate the following matrices:
\vspace{5mm}\\
\vspace{3mm}
\noindent $M^2 = \left(\begin{array}{c|c} M_{11}^2 + M_{12}M_{21} & M_{11}M_{12}+M_{12}M_{22}\\ 
\hline
\noindent M_{21} M_{11}+M_{22}M_{21} & M_{21} M_{12} + M_{22}^2\\ \end{array}\right),$\\
\vspace{3mm}
$(PMP)^2 = \left(\begin{array}{c|c} M_{11}^2 + P_{11}M_{12}M_{21}P_{11} & M_{11}P_{11}M_{12}+P_{11}M_{12}M_{22}\\ 
\hline
M_{21}P_{11} M_{11}+M_{22}M_{21}P_{11} & M_{21}P_{11}P_{11} M_{12} + M_{22}^2\\ \end{array}\right),$\\
\vspace{3mm}
$N_1^2 = \left(\begin{array}{c|c} P_{11}M_{11}P_{11}M_{11} + M_{12}M_{21} & P_{11}M_{11}M_{12}+M_{12}M_{22}\\
\hline
M_{21}P_{11} M_{11}+M_{22}M_{21} & M_{21} M_{12} + M_{22}^2\\ \end{array}\right),$\\
\vspace{3mm}
$N_2^2 = \left(\begin{array}{c|c} P_{11}M_{11}P_{11}M_{11} + P_{11}M_{12}M_{21}P_{11} & P_{11}M_{11}P_{11}M_{12}+P_{11}M_{12}M_{22}\\ 
\hline
M_{21}P_{11}P_{11} M_{11}+M_{22}M_{21}P_{11} & M_{21}P_{11}P_{11} M_{12} + M_{22}^2\\ \end{array}\right).$\\

By Lemma~\ref{MM}, the equation $P_{11}M_{12}M_{21}P_{11}=M_{12}M_{21}$ holds. If $P_{11}M_{12}M_{22}=M_{12}M_{22}$, then comparing the matrices $N_2^2$ and $M^{2}$, or $N_1^2$ and $(PMP)^2$ we have~Theorem~\ref{main11}. 

If $P_{11}M_{11}M_{12}=M_{11}M_{12}$, then comparing the matrices $N_1^2$ and $M^{2}$, or $N_2^2$ and $(PMP)^2$, we have~Theorem~\ref{main12}.
\hfill $\square$ \end{proof} \\

%\noindent {\bf Remark~2.} If $G$ is a strongly regular graph with parameters $(n,k,\lambda,\mu)$, then a sought graph is a strictly Deza %graph if and only if $\lambda \neq 0$, $\mu \neq 0$.\\

\begin{remark}\label{rem1} The combinatorial meaning of the matrix conditions are as follows. \\

\noindent $(1)$ Condition $P_{11}M_{12}M_{22}=M_{12}M_{22}$ from Theorem~\ref{main11} means that for any $v \in V(G)\setminus V(H)$ and for any $x,y \in V(H)$ such that $\varphi(x)=y$, the number of common neighbours for $v$ and $x$ in $G\setminus H$ is equal to the number of common neighbours for $v$ and $y$ in $G\setminus H$;\\

\noindent $(2)$ Condition $P_{11}M_{11}M_{12}=M_{11}M_{12}$ from Theorem~\ref{main12} means that for any $v \in V(G)\setminus V(H)$ and for any $x,y \in  V(H)$ such that $\varphi(x)=y$, the number of common neighbours for $v$ and $x$ in $H$ is equal to the number of common neighbours for $v$ and $y$ in $H$.
\end{remark}

\begin{remark}\label{rem3} If $G=H$ then $N_1=N_2$ and we have the dual Seidel switching from Theorem~\ref{dual}.
\end{remark}

\begin{remark}\label{rem4} If $G$ is a strongly regular graph meeting conditions of Theorem~\ref{main12}, then the theorem holds as well. Moreover, in both Theorems~\ref{main11} and~\ref{main12}, $H$ is a strictly Deza graph if and only if $\lambda \neq 0$ and $\mu \neq 0$.
\end{remark}

\begin{remark} If $G$ is a strongly regular graph meeting conditions of both Theorems~\ref{main11} and~\ref{main12}, then $\lambda=\mu$. 
\end{remark}

Next three examples illustrate Theorem~\ref{main11}. \\

\noindent {\bf Example~1.} The line graph of the complete bipartite graph $K_{m,m}$, where $m \geqslant 2$, is called the {\em lattice graph} $L_2(m)$ with parameters $(m^2,2(m-1),m-2,2)$. 

The lattice graph $G=L_2(m)$, where $m \geqslant 5$,  has the induced subgraph $H=L_2(m-2)$. For odd $m$, there is one Seidel automorphism of $H$ corresponding to the main diagonal lattice symmetry. For even $m$, there are two Seidel automorphisms of $H$ corresponding to the main diagonal lattice symmetry and to the central symmetry. Thus, we have the induced subgraph $H$ with the Seidel automorphism of a strongly regular graph $G$. 

Now let us check the condition of Theorem~\ref{main11}. It is enough to verify that the combinatorial condition (1) of Remark~\ref{rem1} holds. Indeed, for any $v \in V(G)\setminus V(H)$ and for any $x,y \in V(H)$ such that $\varphi(x)=y$, there is the only common neighbour for $v$ and $x$, and for $v$ and $y$ in $G\setminus H$.

Hence, by Theorem~\ref{main11}, there are two non-isomorphic Deza graphs with parameters $(m^2,2(m-1),m-2,2)$ whose strongly regular children are $L_2(m)$ and its compliment. By Theorem~\ref{spec2}, its distinct eigenvalues belong to the set $\{2(m-1),\pm (m-2),\pm 2\}$. \\

\noindent {\bf Example~2.} Let $K_n$ be the complete graph with $n$ vertices on the set $\{1,2,\dots,n\}$, $n\geqslant 2$. The line graph of $K_n$ is called the {\em triangular graph} $T(n)$ with parameters $(n(n-1)/2,2(n-2),n-2,4)$.

Let $n\geqslant 8$ is even. The triangular graph $G=T(n)$ has the induced subgraph $H=T(n-2)$ defined as the line graph of $K_{n-2}$ on the set $\{1,2,\dots,n-2\}$. The vertices of $G$ can be viewed as $2$-subsets of the set $\{1,2,\dots,n\}$. There is a Seidel automorphism of $H$ defined by $\varphi(\{i,j\})=\{n-i-2,n-j-2\}$. 

Now let us check the condition of Theorem~\ref{main11}. If $v=\{n,n-1\}$, then $v$ doesn't have neighbours in $H$, but any vertex from $H$ has four neighbours with $v$. If $v=\{k,l\}$, where $k\in \{1,2,\dots,n-2\}$, $l\in \{n,n-1\}$, is not from $H$, and vertices $x=\{i,j\}$, $y=\varphi (x)$ are from $H$, then out of $H$ a vertex $v$ has exactly one common neighbour with $x$ and exactly one common neighbour with $y$. Thus, condition  of Theorem~\ref{main11} holds, and we have  two non-isomorphic Deza graphs with parameters $(n(n-1)/2,2(n-2),n-2,4)$ for any even integer $n$. Children of this Deza graph are the triangular graph and its complement. By Theorem~\ref{spec2}, all distinct eigenvalues of this Deza graph belong to the set $\{2(n-2),\pm (n-4),\pm 2\}$. \\

\noindent {\bf Example~3.} Let $G$ be the triangular graph $T(7)$. There is an induced subgraph $H = L_2(3)$ in $G$ 
with the diagonal symmetry under conditions of Theorem~\ref{main11}. The matrices $N_1$ and $N_2$ give two non-isomorphic Deza graphs with the same parameters $(21,10,5,4)$ whose spectra are  $\{10^1,3^4,2^3,-2^{11},-3^2\}$ and $\{10^1,3^2,2^6,-2^8,-3^4\}$, where the exponents denote multiplicities.\\

Let $G$ be the triangular graph $T(8)$. There is an induced subgraph $H = L_2(4)$ in $G$ with the diagonal and the central symmetries 
under conditions of Theorem~\ref{main11}. In each of the cases the matrices $N_1$ and $N_2$ correspond to non-isomorphic Deza graphs with the same parameters $(28,12,6,4)$ whose eigenvalues belong to the set $\{12,\pm 4,\pm 2\}$.\\

Next two examples illustrate Theorem~\ref{main12}. \\

\noindent {\bf Example~4.} The lattice graph $G=L_2(m)$ has an induced subgraph $H=K_2\times K_m$. By Remark~\ref{rem4}, Theorem~\ref{main12} holds for $G$ and $H$. For any even $m \geqslant 6$, there is a Seidel automorphism $\varphi$ of $H$ corresponding to the central symmetry of the lattice with two rows and $m$ columns. Thus, we have the induced subgraph $H$ with the Seidel automorphism of a strongly regular graph $G$. The subgraph $H$ meets condition of Theorem~\ref{main12} since the combinatorial condition (2) of Remark~\ref{rem1} holds. Indeed, for any $v \in V(G)\setminus V(H)$ and for any $x,y \in V(H)$ such that $\varphi(x)=y$, there is the only neighbour in $H$ for both pairs $v$ and $x$, and $v$ and $y$. Hence, by Theorem~\ref{main12}, for $m \geqslant 6$ there are two non-isomorphic Deza graphs with parameters $(m^2,2(m-1),m-2,2)$ whose strongly regular children are $L_2(m)$ and its complement. 

The resulting graph $G_1$, which is a strictly Deza graph, has an induced subgraph $H_1 = K_{m-2}\times K_m$ defined on the set $V(G_1)\setminus V(H)$. There is a subgraph $H_2 =K_2\times K_m$ of $H_1$ meeting conditions of Theorem~\ref{main12} for $G_1$ so that a new Deza graph $G_2$ is obtained. Continuing this process, $m/2$ new non-isomorphic Deza graphs are obtained with the same parameters. By Theorem~\ref{spec2}, all distinct eigenvalues of these Deza graphs belong to the set $\{2(m-1),\pm (m-2),\pm 2\}$. \\

\noindent {\bf Example~5.}  Let $G=T(n)$ be the triangular graph for some even $n\geqslant 8$. It has an induced subgraph $H=K_2\times K_{n-2}$ on the neighbourhood $N(w)$ of a vertex $w \in V(G)\setminus V(H)$. This subgraph $H$ has a Seidel automorphism $\varphi$ corresponding to the central symmetry of the lattice. Let us show that the combinatorial condition (2) of Remark~\ref{rem1} holds for $H$. For any $v \in V(G)\setminus V(H)$ and for any two non-adjacent vertices $x,y \in V(H)$ such that $\varphi(x)=y$ consider the following cases. If $v=w$ then the condition holds since $H$ is a regular graph. If $v\neq w$ then there are four subcases: 
\begin{itemize}
    \item $x,y \in N(v) \cap N(w)$; 
    \item $x,y \notin N(v) \cap N(w)$;
    \item $x \in N(v) \cap N(w)$ and $y \notin N(v) \cap N(w)$;
    \item $y \in N(v) \cap N(w)$ and $x \notin N(v) \cap N(w)$,
\end{itemize}
each of which gives $|N(x) \cap N(v) \cap N(w)|=|N(y) \cap N(v) \cap N(w)|=2$. Thus, by Theorem~\ref{main12} we have two non-isomorphic Deza graphs with parameters $(n(n-1)/2,2(n-2),n-2,4)$ for any even integer $n$. Children of this Deza graph are the triangular graph and its complement. By Theorem~\ref{spec2}, all distinct eigenvalues of this Deza graph belong to the set $\{2(n-2),\pm (n-4),\pm 2\}$. 

\section{More constructions}

In this section we present other matrix-based constructions of Deza graphs with strongly regular children. 

\begin{theorem}\label{main2}
Let $M$ be the adjacency matrix of a strongly regular graph $G$ with parameters $(n,k,\lambda,\mu)$ with $\lambda=\mu$. If there exists a fixed point free Seidel automorphism of $G$ and its permutation matrix is $P$, then the matrix $M+P$ is the adjacency matrix of a Deza graph with strongly regular children whose adjacency matrices are $PM$ and $(J-I-PM)$.
\end{theorem}

\proof Since $G$ is a strongly regular graph with parameters $(n,k,\lambda,\mu)$, where $\lambda=\mu$, then by~(\ref{M^2}) we have:
$$M^2=kI+\lambda M +\lambda(J-I-M)=kI+\lambda (J-I).$$ 
Since $P$ is the permutation matrix of a Seidel automorphism of $G$ and $\lambda=\mu$, then $PM$ is the adjacency matrix of a strongly regular graph. Moreover, the equation $$M^2=(PM)^2$$ holds. Since $P$ has no ones on the main diagonal, then the automorphism has no 
fixed points. So, $M+P$ is the adjacency matrix of a graph. Thus, we have: $$(M+P)^2=M^2+2PM+P^2=(PM)^2+2PM+P^2=kI +\lambda(J-I)+2PM+I=$$ $$=(k+1)I+\lambda PM+2PM+\lambda(J-I)-\lambda PM=$$ $$=(k+1)I+(\lambda+2)PM+\lambda(J-I-PM),$$ which gives a Deza graph with strongly regular children whose adjacency matrices are $PM$ and $(J-I-PM)$. This completes the proof. \hfill $\square$\\
 
\noindent {\bf Example~5.} Consider the lattice graph $L_2(4)$ with parameters $(16,6,2,2)$. It has a fixed point free Seidel automorphism corresponding to the central symmetry of the lattice (see Example~1). Let $P$ be its permutation matrix and $M$ be the adjacency matrix of $L_2(4)$. Then by Theorem~\ref{main2}, matrix $P+M$ is the adjacency matrix of a Deza graph with parameters $(16,7,4,2)$. \\

\noindent {\bf Example~6.} Consider the Clebsch graph $G$ with parameters $(16,10,6,6)$ whose adjacency matrix is $M$. It has a fixed point free Seidel automorphism. If $P$ is its permutation matrix, then by Theorem~\ref{main2}, matrix $P+M$ is the adjacency matrix of a Deza graph with parameters $(16, 11,8,6)$. 

\begin{theorem}\label{main1}
Let $M$ be the adjacency matrix of a strongly regular graph $G$ with parameters $(n,k,\lambda,\mu)$. If there exists a fixed point free Seidel automorphism of $G$ and its permutation matrix is $P$, then the matrix $P(M+I)$ is the adjacency matrix of a Deza graph with valency $k+1$, parameters $\{a,b\}=\{\lambda+2,\mu\}$ and with strongly regular children whose adjacency matrices are $M$ and $J-I-M$. 
\end{theorem}

\proof Since $P$ is the permutation matrix of a Seidel automorphism of $G$, we have:
$$(P(M+I))^2 = P(M+I)P(M+I)=(PMP+PIP)(M+I)=(M+I)^2=M^2+2M+I^2.$$
Since $P$ has no ones on the main diagonal, then the automorphism has no fixed points, and $P(M+I)$ is the adjacency matrix of a graph. Thus, by~(\ref{M^2}) we have:
$$(P(M+I))^2 = (k+1)I + (\lambda +2) M +\mu(J-I-M),$$ 
which gives a Deza graph with valency $k+1$, parameters $\{a,b\}=\{\lambda+2,\mu\}$ and with strongly regular children whose adjacency matrices are $M$ and $J-I-M$. This completes the proof. \hfill $\square$\\

\noindent {\bf Example~7.} Consider the lattice graph $L_2(n)$, where $n$ is even, with parameters $(n^2,2(n-1),n-2,2)$. As it was shown in Example~1, central symmetry of $L_2(n)$ can be considered as a fixed point free Seidel automorphism $\varphi$. Thus, if $P$ is the  permutation matrix of $\varphi$, and $M$ is the adjacency matrix of $L_2(n)$, then by Theorem~\ref{main1}, matrix $P(M+I)$ is the adjacency matrix of a Deza graph with parameters $(n^2,2n-1,n,2)$. 

\section*{Acknowledgements} 
The research was funded by RFBR according to the research project 17-51-560008. The work is supported by Mathematical Center in Akademgorodok under agreement No. 075-15-2019-1613 with the Ministry of Science and Higher Education of the Russian Federation. The third author is supported by RFBR according to the research project 20-51-53023. The main results of the paper were obtained during workshops on Algebraic Graph Theory held on November 5-8, 2019, and on March 14-17, 2020, Akademgorodok, Novosibirsk, Russia.

\end{document}